\RequirePackage{etoolbox}
\csdef{input@path}{%
{sty/},
{img/},
{bib/}
}
\documentclass[numbers,compress]{vmsta}

\volume{0}
\issue{0}
\pubyear{2022}
\articletype{research-article}

\newtheorem{thm}{Theorem}
\theoremstyle{definition}
\newtheorem{defin}{Definition}

\newtheorem{lemma}{Lemma}

\newtheorem{assumption}{Assumption}
\newcommand{\eps}{\varepsilon}
\newcommand{\R}{\mathbb{R}}
\newcommand{\pr}{\mathrm{P}}
\newcommand{\E}{\mathrm{E}}
\newcommand{\la}{\lambda}
\newcommand{\al}{\alpha}
\newcommand{\ba}{\begin{align}}
\newcommand{\ea}{\end{align}}

\begin{document}

\begin{frontmatter}

\pretitle{Research Article}

\title{Long-time dynamics of a stochastic density dependent predator-prey model with Holling II functional response and jumps}

\author[a]{\inits{Olg.}\fnms{Olga}~\snm{Borysenko}\thanksref{}\ead[label=e1]{olga\_borisenko@ukr.net}\orcid{0000-0000-0000-0000}}

\author[b]{\inits{O.}\fnms{Oleksandr}~\snm{Borysenko}\thanksref{cor1}\ead[label=e2]{odb@univ.kiev.ua}}
\thankstext[id=cor1]{Corresponding author.}
\address[a]{\institution{Department of Mathematical Physics,
National Technical University of Ukraine}, 37, Prosp.Peremohy, Kyiv, 03056, \cny{ Ukraine
}}

\address[b]{\institution{Department of Probability Theory, Statistics and Actuarial
Mathematics, Taras Shevchenko National University of Kyiv , Ukraine}, 64 Volodymyrska Str., Kyiv, 01601 \cny{Ukraine}}

\begin{abstract}
The existence and uniqueness of a global positive solution is proven for the system
of stochastic differential equations describing a nonautonomous stochastic density dependent predator-prey model with Holling-type II functional response disturbed by white noise, centered and non-centered Poisson noises. Sufficient conditions
are obtained for stochastic ultimate boundedness, stochastic permanence, non-persistence in the mean, weak persistence in the mean and extinction of a population densities in the considered stochastic predator-prey model.
\end{abstract}



\end{frontmatter}

\section{Introduction}\label{}

The deterministic autonomous Rosenzweig-Mac'Arthur model (\cite{Ian}) is a generalization of Voltera-Verhulst model in that the linear functional response is replaced by Holling II functional response. This model has a form
\begin{align}\label{eq1}
dx_1(t)=x_1(t)\left(a_1-bx_1(t)-\frac{cx_2(t)}{1+mx_1(t)}\right)dt,\nonumber \\
dx_2(t)=x_2(t)\left(-a_2+\frac{\kappa cx_1(t)}{1+mx_1(t)}\right)dt,\end{align}
where $x_1(t)$ and $x_2(t)$ are the prey and predator population densities at time $t$, respectively;  $a_1>0$ is the growth rate of prey $x_1$; $b>0$ measures the strength of competition among individuals of species $x_1$; $c$ is the maximum ingestion rate; $m>0$ is the half-saturation; $a_2>0$ is the death rate of predator $x_2$, and $\kappa>0$ is the conversion factor.

In the paper \cite{Liu} the stochastic version of model $(\ref{eq1})$ is considered in the following form
\begin{align}\label{eq2}
dx_1(t)=x_1(t)\left(a_1-bx_1(t)-\frac{cx_2(t)}{1+mx_1(t)}\right)dt+\sigma_1 x_1(t)dw_1(t),\nonumber \\
dx_2(t)=x_2(t)\left(-a_2+\frac{\kappa cx_1(t)}{1+mx_1(t)}\right)dt+\sigma_2 x_2(t)dw_2(t),\end{align}
where $w_1(t)$ and $w_2(t)$ are mutually independent Wiener processes. In \cite{Liu}  the authors proved that there is a unique positive solution to the system $(\ref{eq2})$, deduced the conditions that there is a stationary distribution of the system, which implies that the system is permanent, and obtained the sufficient conditions for the system that is going to be extinct.

Population systems may suffer abrupt environmental perturbations,
such as epidemics, fires, earthquakes, etc. So it is natural to introduce Poisson noises into the population model for
describing such discontinuous systems. There are considerable evidences in nature that predator species may be density dependent. So we need to take into account levels of predator density dependence.

In this paper, we consider the non-autonomous density dependent predator-prey model with Holling-type II functional response, disturbed by white noise and jumps generated by centered and non-centered Poisson measures. So, we take into account not only ``small'' jumps, corresponding to the centered Poisson measure, but also the ``large'' jumps, corresponding to the non-centered Poisson measure. This model is driven by the system of stochastic differential equations
\begin{align}\label{eq3}
dx_i(t)=x_i(t)\left[(-1)^{i-1}\left(a_{i}(t)-\frac{c_{i}(t)x_{3-i}(t)}{1+m(t)x_{1}(t)}\right)-b_{i}(t)x_{i}(t)\right]dt\nonumber\\ +\sigma_i(t)x_i(t)dw_i(t)\!\!
+\!\!\int\limits_{\mathbb{R}}\gamma_i(t,z)x_i(t^{-})\tilde\nu_1(dt,dz)\!+\!\!\int\limits_{\mathbb{R}}\delta_i(t,z)x_i(t^{-})\nu_2(dt,dz),\nonumber\\ x_i(0)=x_{i0}>0,\ i=1,2.
\end{align}
where $x_1(t)$ and $x_2(t)$ are the prey and predator population densities at time $t$, respectively, $b_2(t)>0$ is the predator density dependence rate, $c_2(t)=\kappa c_1(t)$, $w_i(t), i=1,2$ are independent standard one-dimensional Wiener processes, $\nu_i(t,A), i=1,2$ are independent Poisson measures, which are independent on $w_i(t),i=1,2$, $\tilde\nu_1(t,A)=\nu_1(t,A)-t\Pi_1(A)$,  $E[\nu_i(t,A)]=t\Pi_i(A),i=1,2$, $\Pi_i(A), i=1,2$ are finite measures on the Borel sets $A$ in $\mathbb{R}$.

To the best of our knowledge, there have been no papers devoted to the dynamical properties of the stochastic predator-prey model $(\ref{eq3})$, even in the case of centered Poisson noise. It is worth noting that the impact of centered and non-centered Poisson noises to the stochastic non-autonomous logistic model, to the stochastic two-species mutualism model and to the stochastic predator-prey model with modified version of Leslie-Gower term and Holling-type II functional response is studied in the papers \cite{Bor1} -- \cite{Bor4}.

In the following we will use the notations $X(t)=(x_1(t),x_2(t))$, $X_0=(x_{10},x_{20})$, $|X(t)|=\sqrt{x_1^2(t)+x_2^2(t)}$, $\R^2_{+}=\{X\in\R^2:\ x_1>0,x_2>0\}$,
\begin{align}\alpha_i(t)=a_{i}(t)+\int_{\R}\delta_{i}(t,z)\Pi_2(dz),\nonumber\\
\beta_i(t)\!=\!\frac{\sigma^2_i(t)}{2}\!+\!\!\!\int\limits_{\R}[\gamma_i(t,z)\!-\!\ln(1\!+\!\gamma_i(t,z))]\Pi_1(dz)\!
-\!\!\!\int\limits_{\R}\ln(1\!+\!\delta_i(t,z))\Pi_2(dz),
\nonumber
\end{align} $i=1,2$. For the bounded, continuous functions $f(t),f_i(t), t\in[0,+\infty), i=1,2$, let us denote
\begin{align}f_{\sup}=\sup_{t\ge0}f(t), f_{\inf}=\inf_{t\ge0}f(t), f_{i\sup}=\sup_{t\ge0}f_i(t), f_{i\inf}=\inf_{t\ge0}f_i(t), i=1,2,\nonumber\\ f_{\max}=\max\{f_{1\sup},f_{2\sup}\}, f_{\min}=\min\{f_{1\inf},f_{2\inf}\}.\nonumber\end{align}

We prove that the system $(\ref{eq3})$ has a unique, positive, global (no explosion
in a finite time) solution for any positive initial value, and that this solution is stochastically ultimately bounded. The sufficient conditions for stochastic permanence, non-persistence in the mean, weak persistence in the mean and extinction of solution are derived.

The rest of this paper is organized as follows. In Section 2, we prove the existence of the unique global positive solution to the system $(\ref{eq3})$ and derive some auxiliary results. In Section 3, we prove the stochastic ultimate boundedness of the solution to the system $(\ref{eq3})$, obtainig conditions under which the solution is stochastically permanent. The sufficient conditions for non-persistence in the mean, weak persistence in the mean and extinction of the solution are derived.

\section{Existence of global solution and some auxiliary lemmas}

Let $(\Omega,{\cal F},\pr)$ be a probability space, $w_i(t), i=1,2, t\ge0$ are independent standard one-dimensional Wiener processes on $(\Omega,{\cal F},\pr)$, and $\nu_i(t,A), i=1,2$ are independent Poisson measures defined on $(\Omega,{\cal F},\pr)$ independent on $w_i(t), i=1,2$. Here $\E[\nu_i(t,A)]=t\Pi_i(A), i=1,2$, $\tilde\nu_i(t,A)=\nu_i(t,A)-t\Pi_i(A), i=1,2$, $\Pi_i(\cdot), i=1,2$ are finite measures on the Borel sets in $\mathbb{R}$. On the probability space $(\Omega,{\cal F},\pr)$ we consider an increasing, right continuous family of complete sub-$\sigma$-algebras $\{{\cal F}_{t}\}_{t\ge0}$, where ${\cal F}_{t}=\sigma\{w_i(s),\nu_i(s,A), s\le t,i=1,2\}$.

We need the following assumption.

\begin{assumption} \label{ass1} It is assumed, that $a_{i}(t), b_i(t)$, $c_i(t), \sigma_i(t), \gamma_{i}(t,z), \delta_{i}(t,z), i=1,2$, $m(t)$ are bounded, continuous on $t$ functions, $a_{i}(t)>0,i=1,2$, $b_{i\inf}>0$, $c_{i\inf}>0, i=1,2$, $m_{\inf}>0$, and $\ln(1+\gamma_i(t,z)),\ln(1+\delta_i(t,z)), i=1,2$ are bounded, $\Pi_i(\R)<\infty, i=1,2$.\end{assumption}

In what follows we will assume that Assumption \ref{ass1} holds.
\begin{thm} \label{thm1}
There exists a unique global solution $X(t)$ to the system $(\ref{eq3})$ for any initial value $X(0)=X_0\in \R^2_{+}$, and
$\pr\{X(t)\in \R^2_{+}\}=1$.
\end{thm}

\begin{proof}
Let us consider the system of stochastic differential equations
\begin{align}\label{eq4}
d\xi_i(t)=\left[(-1)^{i-1}\left(a_{i}(t)-\frac{c_{i}(t)\exp\{\xi_{3-i}(t)\}}{1+m(t)\exp\{\xi_1(t)\}}\right)-b_{i}(t)\exp\{\xi_{i}(t)\}\right.
\nonumber \\ \left.
-\beta_i(t)\rule{0pt}{14pt}\right]dt+\sigma_i(t)dw_i(t)
+\int\limits_{\mathbb{R}}\ln(1+\gamma_i(t,z))\tilde\nu_1(dt,dz)\nonumber\\+\int\limits_{\mathbb{R}}\ln(1+\delta_i(t,z))\tilde\nu_2(dt,dz),\quad \xi_i(0)=\ln x_{i0},\ i=1,2.
\end{align}
The coefficients of the system $(\ref{eq4})$ are local Lipschitz continuous. So, for any initial value $(\xi_1(0),\xi_2(0))$ there exists a unique local solution $\Xi(t)=(\xi_1(t),\xi_2(t))$ on $[0,\tau_{e})$, where $\sup_{t<\tau_{e}}|\Xi(t)|=+\infty$ (cf. Theorem 6, p.246, \cite{GikhSkor}). Therefore, from the It\^{o}'s formula we derive that the process $X(t)=(\exp\{\xi_1(t)\},\linebreak\exp\{\xi_2(t)\})$ is a unique, positive local solution to the system (\ref{eq3}). To show this solution is global, we need to show that $\tau_{e}=+\infty$ a.s. Let $n_0\in\mathbb{N}$ be sufficiently large for $x_{i0}\in[1/n_0,n_0], i=1,2$. For any $n\ge n_0$ we define the stopping time
\begin{align}
\tau_n=\inf\left\{t\in[0,\tau_e):\ X(t)\notin\left(\frac{1}{n},n\right)\times\left(\frac{1}{n},n\right)\right\}.\nonumber
\end{align}
 It is easy to see that $\tau_n$ is increasing as $n\to+\infty$. Denote $\tau_{\infty}=\lim_{n\to\infty}\tau_n$, whence $\tau_{\infty}\le\tau_e$ a.s. If we prove that $\tau_\infty=\infty$ a.s., then $\tau_e=\infty$ a.s. and $X(t)\in \R^2_{+}$ a.s. for all $t\in[0,+\infty)$. So we need to show that $\tau_\infty=\infty$ a.s. If it is not true, there are constants $T>0$ and $\varepsilon\in(0,1)$, such that $\pr\{\tau_{\infty}<T\}>\varepsilon$. Hence, there is $n_1\ge n_0$ such that
\begin{align}\label{eq5}\pr\{\tau_{n}<T\}>\varepsilon,\quad \forall n\ge n_1.\end{align}
For the non-negative function $V(X)=\sum\limits_{i=1}^2 k_{i}(x_i-1-\ln x_i)$, $X=(x_1,x_2)$, $x_i>0$, $k_{i}>0$, $i=1,2$
by the It\^{o}'s formula we obtain
\begin{align}\label{eq6}
dV(X(t))=\sum_{i=1}^2 k_i \left\{\rule{0pt}{20pt}(-1)^{i-1}(x_i(t)-1)\left(a_i(t)-\frac{c_{i}(t)x_{3-i}(t)}{1+m(t)x_{1}(t)}\right)\right.\nonumber\\ \left.-b_i(t)x_i(t)(x_i(t)-1)+\beta_i(t)+\int\limits_{\R}\delta_i(t,z)x_i(t)\Pi_2(dz)\right\}dt\nonumber\\ \!+\!\sum_{i=1}^2 k_i\!\left\{\! \rule{0pt}{20pt}(x_i(t)\!-\!1)\sigma_i(t)dw_i(t)\right.\!\!
+\!\!\int\limits_{\R}\![\gamma_i(t,z)x_i(t^{-})\!-\!\ln(1\!+\!\gamma_i(t,z))]\tilde\nu_1(dt,dz)\nonumber\\ \left.\!+\!
\int\limits_{\R}[\delta_i(t,z)x_i(t^{-})\!-\!\ln(1\!+\!\delta_i(t,z))]\tilde\nu_2(dt,dz)\right\}.
\end{align}
Let us consider the function
\begin{align}
f(t,x_1,x_2)=-k_1 b_1(t)x_1^2+k_1\left(\rule{0pt}{12pt}\alpha_1(t)+b_1(t)\right)x_1+k_1(\beta_1(t)-a_1(t)) \nonumber\\
-k_2b_2(t)x_2^2+k_2\left(-a_2(t)+b_2(t)+\int_{\R}\delta_2(t,z)\Pi_2(dz)\right)x_2+k_2(\beta_2(t)+a_2(t))\nonumber\\
\nonumber\\ +\frac{k_2(x_2-1)c_2(t)x_1-k_1(x_1-1)c_1(t)x_2}{1+m(t)x_1}, x_i>0, i=1,2.\nonumber
\end{align}
We have estimate
\begin{align}
f(t,x_1,x_2)\le \sum_{i=1}^2 k_i(\rule{0pt}{14pt}-b_{i\inf}x_i^2+(\alpha_{i\sup}+b_{i\sup})x_i)+k_1c_{1\sup}x_2\nonumber\\
+k_1(\beta_{1\sup}-a_{1\inf})+k_2(\beta_{2\sup}+a_{2\sup})+\frac{c_1(t)x_1x_2(k_2\kappa-k_1)}{1+m(t)x_1}\nonumber
\end{align}
So for $k_2=k_1/\kappa$ there is a constant $L(k_1,k_2)>0$, such that $f(t,x_1,x_2)\le L(k_1,k_2)$.
From $(\ref{eq6})$ we obtain by integrating
\begin{align}\label{eq7}
V(X(T\wedge\tau_n))\le V(X_0)+L(k_1,k_2)(T\wedge\tau_n)\nonumber\\+\sum_{i=1}^2k_{i}\left\{\int\limits_0^{T\wedge\tau_n}(x_i(t)-1)\sigma_i(t)dw_i(t)+
\int\limits_0^{T\wedge\tau_n}\!\!\!\int\limits_{\R}\left[\gamma_i(t,z)x_i(t^{-})\!-\!\ln(1\right.\right.\nonumber\\ \left.\left.+\gamma_i(t,z))\right]\tilde\nu_1(dt,dz)+
\int\limits_0^{T\wedge\tau_n}\!\!\!\int\limits_{\R}[\delta_i(t,z)x_i(t^{-})\!-\!\ln(1\!+\!\delta_i(t,z))]\tilde\nu_2(dt,dz)\right\}.
\end{align}
Taking expectations we derive from $(\ref{eq7})$
\begin{align}\label{eq8}
\E\left[V(X(T\wedge\tau_n))\right]\le V(X_0)+L(k_1,k_2)T.
\end{align}
Set $\Omega_n=\{\tau_n\le T\}$ for $n\ge n_1$. Then by $(\ref{eq5})$, $\pr(\Omega_n)=\pr\{\tau_n\le T\}>\varepsilon$, $\forall n\ge n_1$. Note that for every $\omega\in\Omega_n$ at least one of $x_1(\tau_n,\omega)$ and $x_2(\tau_n,\omega)$ equals either $n$ or $1/n$. So
\begin{align}
V(X(\tau_n))\ge\min\{k_1,k_2\} \min\left\{n-1-\ln n,\frac{1}{n}-1+\ln n\right\}.\nonumber
\end{align}
From $(\ref{eq8})$ it follows
\begin{align}
V(X_0)+L(k_1,k_2)T\ge \E[\mathbf{1}_{\Omega_n}V(X(\tau_n))]\nonumber\\ \ge\varepsilon\min\{k_1,k_2\}\min\left\{n-1-\ln n,\frac{1}{n}-1+\ln n\right\},\nonumber
\end{align}
where $\mathbf{1}_{\Omega_n}$ is the indicator function of $\Omega_n$. Letting $n\to\infty$ leads to the contradiction $\infty>V(X_0)+L(k_1,k_2)T=\infty$. This completes the proof of the theorem.
\end{proof}

\begin{lemma}\label{lm1}
The density of the population $x_i(t), i=1,2$ obeys
\begin{align}
\limsup_{t\to\infty}\frac{\ln x_i(t)}{t}\le0,\ i=1,2\qquad\hbox{a.s.}\nonumber
\end{align}
\end{lemma}

\begin{proof} By the It\^{o}'s formula we have for $i=1,2$
\begin{align}\label{eq9}
e^t\ln x_i(t)-\ln x_{i0}=\int\limits_0^te^s\left\{\ln x_i(s)+(-1)^{i-1}\left[a_i(s)-\frac{c_i(s)x_{3-i}(s)}{1+m(s)x_1(s)}\right]\right. \nonumber \\ \left.-b_i(s)x_i(s)-\frac{\sigma_i^2(s)}{2}+\int\limits_{\R}[\ln(1+\gamma_i(s,z))-\gamma_i(s,z)]\Pi_1(dz) \right\}ds+\psi_i(t),
\end{align}
where
\ba
\psi_i(t)=\int\limits_0^te^s\sigma_i(s)dw_i(s)+\int\limits_0^t\!\!\int\limits_{\R}e^s\ln(1+\gamma_i(s,z))\tilde\nu_1(ds,dz)\nonumber \\+
\int\limits_0^t\!\!\int\limits_{\R}e^s\ln(1+\delta_i(s,z))\nu_2(ds,dz),\ i=1,2.\nonumber
\end{align}

By virtue of the exponential inequality (\cite{Bor1}, Lemma 2.2) we have
\ba
\pr\left\{\sup_{0\le t\le T}\zeta_{i}(\mu,t)>\beta\right\}\le e^{-\mu\beta}, \ \forall 0<\mu\le1,\ \beta>0,\ i=1,2\nonumber
\end{align}
where
\ba
\zeta_{i}(\mu,t)=\psi_i(t)-\frac{\mu}{2}\int\limits_0^te^{2s}\sigma_i^2(s)ds-
\frac{1}{\mu}\int\limits_0^t\!\!\int\limits_{\R}\left[(1+\gamma_i(s,z))^{\mu e^s}-1 \right.\nonumber\\ \left.-\mu e^s\ln(1\!+\!\gamma_i(s,z))\rule{0pt}{14pt}\right]\Pi_1(dz)ds\!
-\!\frac{1}{\mu}\int\limits_0^t\!\!\int\limits_{\R}\left[(1\!+\!\delta_i(s,z))^{\mu e^s}\!-\!1\right]\Pi_2(dz)ds,\nonumber
\end{align}
$i=1,2$.
Choosing $T=k\tau, k\in \mathbb{N}, \tau>0, \mu=e^{-k\tau},\beta=\theta e^{k\tau}\ln k$,  $\theta>1$ we get
\ba
\pr\left\{\sup_{0\le t\le k\tau}\zeta_{i}(\mu,t)>\theta e^{k\tau}\ln k\right\}\le \frac{1}{k^{\theta}},\ i=1,2.\nonumber
\end{align}

By the Borel-Cantelli lemma for almost all $\omega\in\Omega$, there is a random integer $k_0(\omega)$, such that for $\forall k\ge k_0(\omega)$ and $0\le t\le k\tau$
\begin{align}\label{eq10}
\psi_i(t)\le \frac{1}{2e^{k\tau}}\int\limits_0^te^{2s}\sigma_i^2(s)ds+
e^{k\tau}\int\limits_0^t\!\!\int\limits_{\R}\left[(1+\gamma_i(s,z))^{e^{s-k\tau}}-1\right.\nonumber\\ \left.- e^{s-k\tau}\ln(1+\gamma_i(s,z))\right]\Pi_1(dz)ds
+e^{k\tau}\int\limits_0^t\!\!\int\limits_{\R}\left[(1+\delta_i(s,z))^{e^{s-k\tau}}\right.\nonumber\\\left.-1\rule{0pt}{14pt}\right]\Pi_2(dz)ds+\theta e^{k\tau}\ln k,\ i=1,2.
\end{align}
Appling the inequality $x^r\le 1+r(x-1)$, $\forall x\ge0$, $0\le r\le1$  with $x=1+\gamma_i(s,z)$, $r=e^{s-k\tau}$, then with $x=1+\delta_i(s,z)$, $r=e^{s-k\tau}$,  we derive from $(\ref{eq10})$ the estimates

\begin{align}\label{eq11}
\psi_i(t)\!\le\! \frac{1}{2e^{k\tau}}\int\limits_0^te^{2s}\sigma_i^2(s)ds\!+\!
\int\limits_0^t\!\!\int\limits_{\R}e^s\left[\gamma_i(s,z)\!-\!\ln(1\!+\!\gamma_i(s,z))\right]\Pi_1(dz)ds\nonumber\\ +
\int\limits_0^t\!\!\int\limits_{\R}e^s\delta_i(s,z)\Pi_2(dz)ds +\theta e^{k\tau}\ln k,\ i=1,2.
\end{align}
So from $(\ref{eq9})$ and $(\ref{eq11})$ we get for $i=1,2$
\ba
e^t\ln x_i(t)\le\ln x_{i0}+\int\limits_0^te^s\left\{\ln x_i(s)+(-1)^{i-1}\left[a_i(s)-\frac{c_i(s)x_{3-i}(s)}{1+m(s)x_1(s)}\right]\right. \nonumber \\ \left.-b_i(s)x_i(s)-\frac{\sigma_i^2(s)}{2}\left(1-e^{s-k\tau} \right)+\int\limits_{\R}\delta_i(s,z)\Pi_2(dz) \right\}ds+\theta e^{k\tau}\ln k\nonumber \\\le \ln x_{i0}+\int\limits_0^te^s[\ln x_i(s)-b_{i\inf}x_i(s)+K_i]ds+\theta e^{k\tau}\ln k\nonumber \\\le \ln x_{i0}+L(e^t-1)+\theta e^{k\tau}\ln k, \forall k\ge k_0(\omega),
0\le t\le k\tau,\nonumber
\end{align}
for some constant $L>0$, where $K_1=\al_{1\sup}, K_2=\sup_{t\ge0}\int_{\R}\delta_2(t,z)\Pi_2(dz)+c_{2\sup}/m_{\inf}$.

So for any $(k-1)\tau\le t\le k\tau$, $\forall k\ge k_0(\omega)$ we have
\ba
\frac{\ln x_i(t)}{\ln t}\le e^{-t}\frac{\ln x_{i0}}{\ln t}+\frac{L}{\ln t}(1-e^{-t})+\frac{\theta e^{k\tau}\ln k}{e^{(k-1)\tau}\ln(k-1)\tau}, i=1,2 \quad \hbox{a.s.}\nonumber
\end{align}
Therefore
\ba
\limsup_{t\to\infty}\frac{\ln x_i(t)}{\ln t}\le\theta e^{\tau},\ i=1,2,\ \forall \theta>1,\ \forall\tau>0, \quad \hbox{a.s.}\nonumber
\end{align}
If $\theta\downarrow1$, $\tau\downarrow0$, then we obtain
\ba
\limsup_{t\to\infty}\frac{\ln x_i(t)}{\ln t}\le 1, \ i=1,2 \quad \hbox{a.s.}\nonumber
\end{align}
So
\ba
\limsup_{t\to\infty}\frac{\ln x_i(t)}{t}\le 0, \ i=1,2 \quad \hbox{a.s.}\nonumber
\end{align}
\end{proof}

\begin{lemma}\label{lm2}
Let $p>0$.Then for any initial value $x_{i0}>0, i=1,2$ we have
\ba
\limsup_{t\to\infty}\E\left[x_i^p(t)\right]\le K_i(p),\ i=1,2,\nonumber
\end{align}
where $K_i(p)>0, i=1,2$ are some constants depending on $p$.
\end{lemma}

\begin{proof} Let $\tau_n$ be the stopping time defined in Theorem \ref{thm1}. Applying the It\^{o}'s formula to the process $V(t,x_i(t))=e^t x_i^p(t)$, $i=1,2, p>0$, we obtain for $i=1,2$
\ba\label{eq12}
V(t\wedge\tau_n,x_i(t\wedge\tau_n))=x_{i0}^p\!\nonumber\\+\!\!\!\!\!\int\limits_0^{t\wedge\tau_n}\!\!\!\!e^sx_i^p(s)\!\!\left\{\!\rule{0pt}{18pt}1\!+\!p\left[(-1)^{i-1}\rule{0pt}{16pt}
\!\!\left(a_i(s)\!-\!\frac{c_i(s)x_{3-i}(s)}{1+m(s)x_1(s)}\right) \left.\rule{0pt}{16pt}-b_i(s)x_i(s)\right]\right.\right.\nonumber\\
+\frac{p(p-1)\sigma_i^2(s)}{2}+\int\limits_{\R}\left[(1+\gamma_i(s,z))^p-1-p\gamma_i(s,z)\right]\Pi_1(dz)\nonumber\\ \left.
+\!\!\!\int\limits_{\R}\!\left[(1\!+\!\delta_i(s,z))^p\!-\!1\right]\Pi_2(dz)\right\}ds\!+\!\!\!\!\int\limits_0^{t\wedge\tau_n}\!\!\!\!p e^sx_i^p(s)\sigma_i(s)dw_i(s)\!\nonumber\\
+\!\!\!\!\int\limits_0^{t\wedge\tau_n}\!\!\!\!\int\limits_{\R}\!\!e^sx_i^p(s^{-})\left[(1\!+\!\gamma_i(s,z))^p\!-\!1\right]\tilde\nu_1(ds,dz)\nonumber\\
+\int\limits_0^{t\wedge\tau_n}\!\!\!\int\limits_{\R}e^sx_i^p(s^{-})\left[(1+\delta_i(s,z))^p-1\right]\tilde\nu_2(ds,dz).
\end{align}
Under Assumption \ref{ass1} there are a constants $K_i(p)>0, i=1,2$, such that
\ba\label{eq13}
e^sx_i^p(s)\left\{\rule{0pt}{18pt}1\!+\!p\left[(-1)^{i-1}\rule{0pt}{16pt}
\left(a_i(s)\!-\!\frac{c_i(s)x_{3-i}(s)}{1+m(s)x_1(s)}\right)-b_i(s)x_i(s)\right]\right.\nonumber\\ \left. +\frac{p(p-1)\sigma_i^2(s)}{2}
+\int\limits_{\R}\left[(1+\gamma_i(s,z))^p-1-p\gamma_i(s,z)\right]\Pi_1(dz)\right.\nonumber\\ \left.
+\!\!\int\limits_{\R}\!\left[(1\!+\!\delta_i(s,z))^p\!-\!1\right]\Pi_2(dz)\right\}\le e^s K_i(p)
\end{align}
From $(\ref{eq12})$ and $(\ref{eq13})$, taking the expectation, we obtain
\ba
\E\left[V(t\wedge\tau_n,x_i(t\wedge\tau_n))\right]\le x_{i0}^p+K_i(p)e^t, i=1,2.\nonumber
\end{align}
If $n\to\infty$, then we get
\ba
e^t\E\left[x_i^p(t)\right]\le x_{i0}^p+K_i(p)e^t,\ i=1,2.\nonumber
\end{align}
Hence
$
\limsup_{t\to\infty}\E\left[x_i^p(t)\right]\le K_i(p),\ i=1,2$.
\end{proof}

\begin{lemma}\label{lm3}
If $p_{2\inf}>0$, where $p_2(t)=-a_2(t)-\beta_2(t)$, then for any initial value $x_{20}>0$, the predator population density $x_2(t)$  has the property that
\begin{align}\label{eq14}
\limsup_{t\to\infty}\E\left[\left(\frac{1}{x_2(t)}\right)^{\theta}\right] \le K(\theta),\ 0<\theta<1,
\end{align}
\end{lemma}

\begin{proof}
For the process $U(t)=1/x_2(t)$ by the It\^{o}'s formula we have
\begin{align}
U(t)=U(0)+\int\limits_0^t U(s) \left[\rule{0pt}{20pt}a_2(s)-\frac{c_2(s)x_{1}(s)}{1+m(s)x_{1}(s)}+b_2(s)x_2(s)+\sigma_2^2(s)\right.\nonumber\\ \left.+\int\limits_{\R}\frac{\gamma_2^2(s,z)}{1+\gamma_2(s,z)}\Pi_1(dz)\right]ds
-\int\limits_0^tU(s)\sigma_2(s)dw_2(s)\nonumber\\-\int\limits_0^t\!\!\!\int\limits_{\R}U(s^{-})\frac{\gamma_2(s,z)}{1+\gamma_2(s,z)}\tilde\nu_1(ds,dz)-
\int\limits_0^t\!\!\!\int\limits_{\R}U(s^{-})\frac{\delta_2(s,z)}{1+\delta_2(s,z)}\nu_2(ds,dz).
\nonumber
\end{align}

Then by the It\^{o}'s formula we derive for $0<\theta<1$
\begin{align}
(1+U(t))^{\theta}=(1+U(0))^{\theta}+\int\limits_0^t \theta(1+U(s))^{\theta-2}\left\{\rule{0pt}{20pt}(1+U(s))U(s)\right.\nonumber\\ \times\left[a_2(s)-\frac{c_2(s)x_{1}(s)}{1+m(s)x_{1}(s)}+b_2(s)x_2(s)\sigma_2^2(s)\!
+\!\int\limits_{\R}\frac{\gamma_2^2(s,z)}{1+\gamma_2(s,z)}\Pi_1(dz)\right]
\nonumber
\end{align}

\begin{align}
+\frac{\theta-1}{2}U^2(s)\sigma_2^2(s)\nonumber\\+\frac{1}{\theta}\int\limits_{\R}\left[(1+U(s))^2\left(\left(\frac{1+U(s)+
\gamma_2(s,z)}{(1+\gamma_2(s,z))(1+U(s))}\right)^{\theta}-1\right)\right.\nonumber\\+\left.
\theta(1+U(s))\frac{U(s)\gamma_2(s,z)}{1+\gamma_2(s,z)}\right]\Pi_1(dz)\nonumber\end{align}

\begin{align} \left.
+\frac{1}{\theta}\int\limits_{\R}(1+U(s))^2\left[\left(\frac{1+U(s)+\delta_2(s,z)}{(1+\delta_2(s,z))(1+U(s))}\right)^{\theta}
-1\right]\Pi_2(dz)\right\}ds\nonumber\\
-\int\limits_0^t\theta(1+U(s))^{\theta-1}U(s)\sigma_2(s)dw_2(s)\nonumber\\+
\int\limits_0^t\!\!\int\limits_{\R}\left[\left(1+\frac{U(s^{-})}{1+\gamma_2(s,z)}\right)^{\theta}\!-\!(1\!+\!U(s^{-}))^\theta\right]\tilde\nu_1(ds,dz)\nonumber
\end{align}

\begin{align}
+\int\limits_0^t\!\!\int\limits_{\R}\left[\left(1+\frac{U(s^{-})}{1\!+\!\delta_2(s,z)}\right)^{\theta}\!-\!(1\!+\!U(s^{-}))^\theta\right]\tilde\nu_2(ds,dz)=
(1+U(0))^{\theta}\nonumber
\end{align}
\begin{align}\label{eq15} +\int\limits_0^t \theta(1+U(s))^{\theta-2}J(s)ds-I_{1,stoch}(t)+I_{2,stoch}(t)+I_{3,stoch}(t),
\end{align}
where $I_{j,stoch}(t), j=\overline{1,3}$ are corresponding stochastic integrals in $(\ref{eq15})$. Under Assumption \ref{ass1} there exist constants $|K_1(\theta)|<\infty$, $|K_2(\theta)|<\infty$ such, that for the process $J(t)$ we have the estimate
\begin{align}
J(t)\le (1+U(t))U(t)\left[a_{2}(t)+b_2(t)U^{-1}(t)+\sigma_2^2(t)\!\!\phantom{\int\limits_{\R}}\right.\nonumber\\ \left.+\int\limits_{\R}\gamma_2(s,z)\Pi_1(dz)\right]
+\frac{\theta-1}{2}U^2(s)\sigma_2^2(s)\nonumber\\ +\frac{1}{\theta}\int\limits_{\R}(1+U(s))^2\left[\left(\frac{1}{1+\gamma_2(s,z)}+
\frac{1}{1+U(s)}\right)^{\theta}-1\right]\Pi_1(dz)
\nonumber\\+\frac{1}{\theta}\int\limits_{\R}(1+U(s))^2\left[\left(\frac{1}{1+\delta_2(s,z)}+
\frac{1}{1+U(s)}\right)^{\theta}-1\right]\Pi_2(dz)\nonumber
\end{align}
\begin{align}
 \le U^2(t)\left[a_{2}(t)+\frac{\sigma_2^2(t)}{2} +\int\limits_{\R}\gamma_2(t,z)\Pi_1(dz)+\frac{\theta}{2}\sigma_2^2(t)\right.\nonumber\\ \left.+\frac{1}{\theta}\int\limits_{\R}[(1+\gamma_2(t,z))^{-\theta}-1]\Pi_1(dz)
+\frac{1}{\theta}\int\limits_{\R}[(1+\delta_2(t,z))^{-\theta}-1]\Pi_2(dz)\right]\nonumber\\ +K_1(\theta)U(t)+K_2(\theta)=-K_0(t,\theta)U^2(t)+K_1(\theta)U(t)+K_2(\theta).
\nonumber
\end{align}
Here we used the inequality $(x+y)^{\theta}\le x^{\theta}+\theta x^{\theta-1}y$, $0<\theta<1$, $x,y>0$, in the first integral term for
$x=(1+\gamma_2(s,z))^{-1}$, $y=(1+U(s))^{-1}$, and then in the second integral term for
$x=(1+\delta_2(s,z))^{-1}$, $y=(1+U(s))^{-1}$. Due to
\begin{align}
\lim_{\theta\to0+}\left[\frac{\theta}{2}\sigma_i^2(t)+\frac{1}{\theta}\int\limits_{\R}[(1+\gamma_i(t,z))^{-\theta}-1]\Pi_1(dz)\right.\nonumber\\ \left.
+\frac{1}{\theta}\int\limits_{\R}[(1+\delta_i(t,z))^{-\theta}-1]\Pi_2(dz)
+\int\limits_{\R}\ln(1+\gamma_i(t,z))\Pi_1(dz)\right.\nonumber\\\left.+\int\limits_{\R}\ln(1+\delta_i(t,z))\Pi_2(dz)\right]=\lim_{\theta\to0+}\Delta(\theta,t)=0,\nonumber
\end{align}
and the condition $p_{2\inf}>0$ we can choose a sufficiently small $0<\theta<1$ so that
\begin{align}
K_0(\theta)=\inf_{t\ge0}K_0(t,\theta)=\inf_{t\ge0}[p_2(t)-\Delta(\theta,t)]>0\nonumber
\end{align}
is satisfied.
So from $(\ref{eq15})$ and the estimate for $J(t)$ we derive
\begin{align}\label{eq16}
d\left[(1+U(t))^\theta\right]\le\theta(1+U(t))^{\theta-2}[-K_0(\theta)U^2(t)+K_1(\theta)U(t)+K_2(\theta)]dt\nonumber\\
-\theta(1+U(t))^{\theta-1}U(t)\sigma_2(t)dw_2(t)+
\int\limits_{\R}\left[\left(1+\frac{U(t^{-})}{1+\gamma_2(t,z)}\right)^{\theta}\right.
\nonumber\\ \left.
\!-(1\!+\!\!U(t^{-}))^\theta\rule{0pt}{18pt}\right]\!\tilde\nu_1(dt,dz)\!\! +\!\!\!
\int\limits_{\R}\!\left[\left(1\!+\!\frac{U(t^{-})}{1\!+\!\delta_2(t,z)}\right)^{\theta}\!\!\!\!
-\!(1\!+U(t^{-}))^\theta\right]\!\tilde\nu_2(dt,dz).
\end{align}
By the It\^{o}'s formula and $(\ref{eq16})$ we have
\begin{align}\label{eq17}
d\left[e^{\la t}(1+U(t))^\theta\right]=\la e^{\la t}(1+U(t))^\theta dt+e^{\la t} d\left[(1+U(t))^\theta\right]\nonumber \\
\le e^{\la t}\theta(1+U(t))^{\theta-2}\left[-\left(K_0(\theta)-\frac{\la}{\theta}\right)U^2(t)+\left(K_1(\theta)+\frac{2\la}{\theta}\right)U(t)\right.\nonumber
\\ \left.
+K_2(\theta)+\frac{\la}{\theta}\right]dt-\theta e^{\la t}(1+U(t))^{\theta-1}U(t)\sigma_2(t)dw_2(t)\nonumber\end{align}
\begin{align}
+e^{\la t}\int\limits_{\R}\left[\left(1+\frac{U(t^{-})}{1+\gamma_2(t,z)}\right)^{\theta}-(1+U(t^{-}))^\theta\right]\tilde\nu_1(dt,dz)\nonumber\\ \displaystyle+
e^{\la t}\int\limits_{\R}\left[\left(1+\frac{U(t^{-})}{1+\delta_2(t,z)}\right)^{\theta}-(1+U(t^{-}))^\theta\right]\tilde\nu_2(dt,dz).
\end{align}

Let us choose $\la=\la(\theta)>0$ such, that $K_0(\theta)-\la/\theta>0$. Then there is a constant $K>0$, such that
\begin{align}\label{eq18}
(1+U(t))^{\theta-2}\left[-\left(K_0(\theta)-\frac{\la}{\theta}\right)U^2(t)\right.\nonumber\\\left.
+\left(K_1(\theta)+\frac{2\la}{\theta}\right)U(t)
+K_2(\theta)+\frac{\la}{\theta}\right]\le K.
\end{align}
Let $\tau_n$ be the stopping time defined in the Theorem \ref{thm1}. Then by integrating $(\ref{eq17})$, using $(\ref{eq18})$  and taking the expectation we obtain
\begin{align}
\E\left[e^{\la (t\wedge\tau_n)}(1+U(t\wedge\tau_n))^\theta\right]\le \left(1+\frac{1}{x_{20}}\right)^\theta+\frac{\theta}{\la}K\left(e^{\la t}-1\right).
\nonumber
\end{align}
Letting $n\to\infty$ leads to the estimate
\begin{align}\label{eq19}
e^t\E\left[(1+U(t))^\theta\right]\le \left(1+\frac{1}{x_{20}}\right)^\theta+\frac{\theta}{\la}K\left(e^{\la t}-1\right).
\end{align}
From $(\ref{eq19})$ we obtain
\begin{align}
\limsup_{t\to\infty}\E\left[\left(\frac{1}{x_2(t)}\right)^{\theta}\right]=\limsup_{t\to\infty}\E\left[U^{\theta}(t)\right]\nonumber\\ \le \limsup_{t\to\infty}\E\left[(1+U(t))^{\theta}\right]\le \frac{\theta}{\la(\theta)}K,\nonumber
\end{align}
this implies  $(\ref{eq14})$.
\end{proof}

\section{The long time behaviour}

\begin{defin}\label{def1}(\cite{LiMao})
The solution $X(t)$ to the system $(\ref{eq3})$ will be said stochastically ultimately
bounded, if for any $\varepsilon\in(0,1)$, there is a positive constant $\chi=\chi(\varepsilon)>0$, such that for
any initial value $X_0\in\R^2_{+}$, the solution to the system $(\ref{eq3})$ has the property that
\begin{align}
\limsup_{t\to\infty}\pr\left\{|X(t)|>\chi\right\}<\varepsilon.\nonumber
\end{align}
\end{defin}

In what follows in this section we will assume that Assumption \ref{ass1} holds.

\begin{thm}\label{thm2}
The solution $X(t)$ to the system $(\ref{eq3})$ is stochastically ultimately
bounded for any initial value $X_0\in\R^2_{+}$.
\end{thm}

\begin{proof}

From Lemma \ref{lm2} we have estimate
\begin{align}\label{eq32}
\limsup_{t\to\infty}E[x_i(t)]\le K_i,\ i=1,2.
\end{align}
For $X=(x_1,x_2)\in\R^2_{+}$ we have $|X|\le x_1+x_2$, therefore, from $(\ref{eq32})$\linebreak
$\limsup_{t\to\infty}E[|X(t)|]\le L=K_1+K_2$. Let $\chi>L/\varepsilon$, $\forall\varepsilon\in(0,1)$. Then applying the Chebyshev inequality yields
\begin{align}
\limsup_{t\to\infty}\pr\{|X(t)|>\chi\}\le\frac{1}{\chi}\lim\sup_{t\to\infty}E[|X(t)|]\le \frac{L}{\chi}<\varepsilon.\nonumber
\end{align}
\end{proof}

The property of stochastic permanence is important since it means the long-time survival in a population dynamics.

\begin{defin}
The population density $x(t)$ will be said stochastically permanent if for any $\eps>0$, there are positive constants $H=H(\eps)$, $h=h(\eps)$ such that
\begin{align}
\liminf\limits_{t\to\infty}\pr\{x(t)\le H\}\ge1-\eps,\quad \liminf\limits_{t\to\infty}\pr\{x(t)\ge h\}\ge1-\eps,
\nonumber
\end{align}
for any inial value $x(0)=x_0>0$.
\end{defin}

\begin{thm}
 If $p_{2\inf}>0$, where $p_2(t)=-a_2(t)-\beta_2(t)$, then for any initial value $x_{20}>0$, the predator population density $x_2(t)$ is stochastically permanent.
\end{thm}

\begin{proof}
From Lemma \ref{lm2} we have estimate
\begin{align}
\limsup_{t\to\infty}E[x_2(t)]\le K.\nonumber
\end{align}
Thus for any given $\eps > 0$, let $H = K/\eps$, by virtue of Chebyshev's inequality, we can derive that
\begin{align}
\limsup\limits_{t\to\infty}\pr\{x_2(t)\ge H\}\le \frac{1}{H}\limsup_{t\to\infty}E[x_2(t)]\le\eps.\nonumber
\end{align}
Consequently $\liminf\limits_{t\to\infty}\pr\{x_2(t)\le H\}\ge1-\eps$.

From Lemma \ref{lm3} we have the estimate
\begin{align}
\limsup_{t\to\infty}\E\left[\left(\frac{1}{x_2(t)}\right)^{\theta}\right] \le K(\theta),\ 0<\theta<1.\nonumber
\end{align}
For any given $\eps > 0$, let $h = (\eps/K(\theta))^{1/\theta}$, then by Chebyshev's inequality, we have
\begin{align}
\limsup\limits_{t\to\infty}\pr\{x_2(t)< h\}\le \limsup\limits_{t\to\infty}\pr\left\{\left(\frac{1}{x_2(t)}\right)^{\theta}>h^{-\theta}\right\}\nonumber\\\le
h^\theta\limsup\limits_{t\to\infty}\E\left[\left(\frac{1}{x_2(t)}\right)^{\theta}\right]\le\eps.\nonumber
\end{align}
Consequently $\liminf\limits_{t\to\infty}\pr\{x_2(t)\ge h\}\ge1-\eps$.
\end{proof}
If the predator is absent, i.e. $x_2(t)\equiv0$ a.s., then the equation for prey population is the non-autonomous stochastic logistic differential equation. The sufficient conditions for the stochastic permanence of the solution to the stochastic logistic differential equation are obtained in \cite{Bor4}:
\begin{thm}\label{thm3}
 If the predator is absent, i.e. $x_2(t)\equiv0$ a.s., and $p_{1\inf}>0$, where $p_1(t)=a_1(t)-\beta_1(t)$, then for any initial value $x_{10}>0$, the prey population density $x_1(t)$ is stochastically permanent.
\end{thm}

\begin{defin}
The solution $X(t)=(x_1(t),x_2(t))$, $t\ge0$ to the system $(\ref{eq3})$ will be said extinct if for every initial data $X_0\in \R^2_{+}$, we have $\lim_{t\to\infty}x_i(t)=0$ a.s., $i=1,2$.
\end{defin}

\begin{thm}\label{thm5}
 If
$
{\bar q}^{*}_i=\limsup_{t\to\infty}\frac{1}{t}\int\limits_0^tq_{i}(s)ds<0$, {where}\begin{align}
q_{1}(t)=a_{1}(t)-\beta_{1}(t),\ q_2(t)=-a_{2}(t)+\frac{c_2(t)}{m(t)}-\beta_{2}(t) \nonumber
\end{align}
then the solution $X(t)$ to the system $(\ref{eq3})$ with the initial condition $X_0\in\R^2_{+}$ will be extinct.
\end{thm}

\begin{proof}
By the It\^{o}'s formula, we have
\begin{align}\label{eq21}
\ln x_i(t)=\ln x_{i0}+\int_0^t\left\{(-1)^{i-1}\left[a_{i}(s)-\frac{c_i(s)x_{3-i}(s)}{1+m(s)x_1(s)}\right]\right.\nonumber\\ \left.-\beta_i(s)-b_i(s)x_i(s)\rule{0pt}{14pt}\right\}ds\!+\!M_i(t)\le \ln x_{i0}\!+\!\int_0^t q_i(s)ds\!+\!M_i(t), i=1,2
\end{align}
where the martingale
\begin{align}\label{eq22}
M_i(t)=\int\limits_0^t\sigma_i(s)dw_i(s)+\int\limits_0^t\!\!\int\limits_{\R}\ln(1+\gamma_i(s,z))\tilde\nu_1(ds,dz) \nonumber\\+
\int\limits_0^t\!\!\int\limits_{\R}\ln(1+\delta_i(s,z))\tilde\nu_2(ds,dz),\ i=1,2,
\end{align}
has quadratic variation (Meyer's angle bracket process)
\begin{align}
\langle M_i,M_i\rangle(t)=\int\limits_0^t\sigma^2_i(s)ds+\int\limits_0^t\!\!\int\limits_{\R}\ln^2(1+\gamma_i(s,z))\Pi_1(dz)ds\nonumber\\+
\int\limits_0^t\!\!\int\limits_{\R}\ln^2(1+\delta_i(s,z))\Pi_2(dz)ds\le Kt, \ i=1,2.\nonumber
\end{align}

Then the strong law of large numbers for local martingales (\cite{Lip}) yields $\lim_{t\to\infty}M_i(t)/t=0, i=1,2$ a.s. Therefore, from $(\ref{eq21})$ we obtain
$$
\limsup_{t\to\infty}\frac{\ln x_i(t)}{t}\le\limsup_{t\to\infty}\frac{1}{t}\int\limits_0^t q_{i}(s)ds<0, \quad\hbox{a.s.}
$$
So $\lim_{t\to\infty}x_i(t)=0, i=1,2$ a.s.
\end{proof}

\begin{defin}(\cite{Liu1})
The solution $X(t)=(x_1(t),x_2(t))$, $t\ge0$ to the system $(\ref{eq3})$ will be said non-persistent in the mean if for every initial data $X_0\in\R^2_{+}$, we have
$\lim_{t\to\infty}\frac{1}{t}\int\limits_{0}^{t}x_i(s)ds=0$ a.s., $i=1,2$.
\end{defin}

\begin{thm}
If ${\bar q}_i^{*}=0$, $i=1,2$, then the solution $X(t)=(x_1(t),x_2(t))$ to the system $(\ref{eq3})$ with the initial value $X_0\in\R^2_{+}$ will be non-persistent in the mean.
\end{thm}

\begin{proof}
From the first equality in $(\ref{eq21})$ we have
\begin{align}\label{eq23}
\ln x_i(t)\le \ln x_{i0}+\int\limits_0^t q_{i}(s)ds
-b_{i\inf}\int\limits_0^tx_i(s)ds+M_i(t),\ i=1,2,
\end{align}
where martingales $M_i(t), i=1,2$ are defined in $(\ref{eq22})$. From the definition of ${\bar q}^{*}_i, i=1,2$ and the strong law of large numbers for $M_i(t), i=1,2$ it follows, that $\forall \eps>0$, $\exists t_0\ge0$, and $\exists \Omega_{\eps}\subset\Omega$, $\pr(\Omega_{\eps})\ge 1-\eps$ such that
\ba
\frac{1}{t}\int\limits_0^t q_{i}(s)ds\le{\bar q}^{*}_i+\frac{\eps}{2},\ \frac{M_i(t)}{t}\le\frac{\eps}{2},\ i=1,2,\ \forall t\ge t_0,\ \omega\in\Omega_{\eps}.
\nonumber
\end{align}

So, from $(\ref{eq23})$ we derive
\begin{align}\label{eq24}
\ln x_i(t)- \ln x_{i0}\le t({\bar q}^{*}_i+\eps)-b_{i\inf}\int\limits_0^tx_i(s)ds\nonumber\\=t\eps-b_{i\inf}\int\limits_0^tx_i(s)ds,\ i=1,2,\ \forall t\ge t_0,\ \omega\in\Omega_{\eps}.
\end{align}
Let $y_i(t) =\int_0^tx_i(s)ds, i=1,2$, then from $(\ref{eq24})$ we have for $i=1,2$
\begin{align}
\ln\left(\frac{dy_i(t)}{dt}\right)\le\eps t-b_{i\inf} y_i(t)+\ln x_{i0},\ \forall t\ge t_0,\ \omega\in\Omega_{\eps}.\nonumber\end{align}
Therefore
\ba
e^{b_{i\inf}y_i(t)}\frac{dy_i(t)}{dt}\le x_{i0}e^{\eps t},\ i=1,2,\ \forall t\ge t_0,\ \omega\in\Omega_{\eps}.
\nonumber
\end{align}
By integrating the last inequality from $t_0$ to $t$ we obtain
\begin{align}
e^{b_{i\inf}y_i(t)}\le \frac{b_{i\inf}x_{i0}}{\eps}\left(e^{\eps t}-e^{\eps t_0}\right)
+e^{b_{i\inf}y_i(t_0)},\ i=1,2,\ \forall t\ge t_0,\ \omega\in\Omega_{\eps}.
\nonumber
\end{align}
So
\begin{align}
y_i(t)\le \frac{1}{b_{i\inf}}\ln\left[e^{b_{i\inf}y_i(t_0)}+\frac{b_{i\inf}x_{i0}}{\eps}\left(e^{\eps t}-e^{\eps t_0}\right)
\right],\ i=1,2, \ \forall t\ge t_0,\ \omega\in\Omega_{\eps},
\nonumber
\end{align}
and therefore
\begin{align}
\limsup_{t\to\infty}\frac{1}{t}\int\limits_0^tx_i(s)ds\le\frac{\eps}{b_{i\inf}},\ i=1,2, \ \forall  \omega\in\Omega_{\eps}.
\nonumber
\end{align}
Since $\eps>0$ is arbitrary and $x_i(t)>0, i=1,2$ a.s., we have for $i=1,2$
\begin{align}\lim_{t\to\infty}\frac{1}{t}\int\limits_{0}^{t}x_i(s)ds=0\ a.s.\nonumber
\end{align}
\end{proof}

\begin{defin} (\cite{Liu1})
The population $x_i(t), i=1,2$ will be said weakly persistent in the mean if for every initial data $x_{i0}>0, i=1,2$, we have
\begin{align}
\limsup_{t\to\infty}\frac{1}{t}\int\limits_{0}^{t}x_i(s)ds>0,\ a.s.,\ i=1,2\nonumber
\end{align}
\end{defin}

\begin{thm}\label{thm7}
If ${\bar p}_2^{*}>0$, where $p_2(t)=-a_2(t)-\beta_2(t)$, then the predator population density $x_2(t)$ with the initial condition $x_{20}>0$ will be weakly persistence in the mean
\begin{align}
{\bar x}_2^{*}=\limsup_{t\to\infty}\frac{1}{t}\int_0^tx_2(s)ds>0\ \hbox{a.s.}
\nonumber
\end{align}
\end{thm}

\begin{proof}
If the assertion of theorem is not true, then $\pr\{\omega\in\Omega\ |\ {\bar x}_2^{*}=0\}>0$. From the first equality in $(\ref{eq21})$ we get
\begin{align}
\frac{1}{t}(\ln x_2(t)-\ln x_{20})+\frac{1}{t}\int_0^tb_2(s)x_2(s)ds=\frac{1}{t}\int_0^tp_2(s)ds\nonumber \\+\frac{1}{t}\int_0^t\frac{c_2(s)x_1(s)}{1+m(s)x_1(s)}ds+\frac{M_2(t)}{t}
\ge \frac{1}{t}\int_0^t p_2(s)ds+\frac{M_2(t)}{t},\nonumber
\end{align}
where martingale $M_2(t)$ is defined in $(\ref{eq22})$. For $\forall\omega\in\{\omega\in\Omega|\ {\bar x}_2^{*}=0\}$ in virtue of the strong law of large numbers for martingale $M_2(t)$ we have
\begin{align}
\limsup_{t\to\infty}\frac{\ln x_2(t)}{t}\ge {\bar p}_2^{*}>0.\nonumber
\end{align}
Therefore
\begin{align}
\pr\left\{\omega\in\Omega|\ \limsup_{t\to\infty}\frac{\ln x_2(t)}{t}>0\right\}>0.
\nonumber
\end{align}
But from Lemma \ref{lm1} we have
\begin{align}
\pr\left\{\omega\in\Omega\ |\ \limsup_{t\to\infty}\frac{\ln x_2(t)}{t}\le0\right\}=1.\nonumber
\end{align}
This is a contradiction.
\end{proof}

\begin{thm}
If ${\bar p}_1^{*}>0$ and ${\bar q}_2^{*}<0$, then the prey population density $x_1(t)$ with initial condition $x_{10}>0$ will be weakly persistence in the mean
\begin{align}
{\bar x}_1^{*}=\limsup_{t\to\infty}\frac{1}{t}\int_0^tx_1(s)ds>0\ \hbox{a.s.}
\nonumber
\end{align}
\end{thm}

\begin{proof}
Let $\pr\{{\bar x}_1^{*}=0\}>0$. From the first equality in $(\ref{eq21})$ we get
\begin{align}\label{eq25}
\frac{1}{t}(\ln x_1(t)-\ln x_{10})+\frac{1}{t}\int_0^t b_1(s)x_1(s)ds=\frac{1}{t}\int_0^tp_1(s)ds\nonumber\\-\frac{1}{t}\int_0^t\frac{c_1(s)x_2(s)}{m(s)+x_1(s)}ds+\frac{M_1(t)}{t}\nonumber\\
\ge \frac{1}{t}\int_0^tp_1(s)ds-\frac{c_{1\sup}}{m_{\inf}t}\int_0^tx_2(s)ds+\frac{M_1(t)}{t}
\end{align}
where martingale $M_1(t)$ is defined in $(\ref{eq22})$. In virtue of the strong law of large numbers for martingale $M_1(t)$ and the result of Theorem \ref{thm5} for the predator population density $x_2(t)$ we obtain from $(\ref{eq25})$
\begin{align}
\limsup_{t\to\infty}\frac{\ln x_1(t)}{t}\ge {\bar p}_1^{*}>0\nonumber
\end{align}
for $\forall\omega\in\{\omega\in\Omega|\ {\bar x}_2^{*}=0\}$.
Therefore
\begin{align}
\pr\left\{\omega\in\Omega|\ \limsup_{t\to\infty}\frac{\ln x_1(t)}{t}>0\right\}>0.
\nonumber
\end{align}
But from Lemma \ref{lm1} we have
\begin{align}
\pr\left\{\omega\in\Omega|\ \limsup_{t\to\infty}\frac{\ln x_1(t)}{t}\le0\right\}=1.
\nonumber
\end{align}
Therefore we have a contradiction.
\end{proof}










\end{document}